\documentclass[12pt]{amsart}
\usepackage{amssymb,amsmath,graphicx,verbatim,eufrak,amsthm}
\usepackage{color}
\usepackage{xcolor}
\usepackage{mdframed}

\usepackage{hyperref}
\hypersetup{
    colorlinks=true,       
    linkcolor=red,          
    citecolor=green,        
    filecolor=magenta,      
    urlcolor=cyan           
}

\newcommand{\thmc}[1]{{\color{#1} $\bullet$}}

\newtheorem{thm}{\thmc{blue} Theorem}
\newtheorem{cor}[thm]{\thmc{cyan} Corollary}
\newtheorem{lem}[thm]{\thmc{cyan} Lemma}

\newtheorem{prop}[thm]{\thmc{cyan} Proposition}

\newtheorem{ques}[thm]{Question}
\newtheorem{conj}[thm]{Conjecture}

\newtheorem*{clm*}{Claim}

\theoremstyle{definition}
\newtheorem{dfn}[thm]{\thmc{red} Definition}
\newtheorem{exm1}[thm]{Example}

\theoremstyle{remark}
\newtheorem{rem}[thm]{\thmc{violet} Remark}

\newenvironment{lem*}[1]{\vspace{1ex}\noindent
{\bf Lemma* (#1).} [restatement]  \hspace{0.5em} \em }{ }
\newenvironment{thm*}[1]{\vspace{1ex}\noindent
{\bf Theorem* (#1).} [restatement]  \hspace{0.5em} \em }{ }

\newenvironment{exm}%
{\begin{exm1}}%
{\hfill$\bigtriangleup\bigtriangledown\bigtriangleup$
\end{exm1} }

\newcommand{\IP}[1]{\left\langle #1 \right\rangle}
\newcommand{\iP}[1]{\langle #1 \rangle}

\newcommand{\set}[1]{\left\{#1\right\}}

\newcommand{\sr}[1]{\left(#1\right)}

\newcommand{\Integer}{\mathbb{Z}}

\newcommand{\Z}{\Integer}

\newcommand{\R}{\mathbb{R}}

\newcommand{\eps}{\varepsilon}

\newcommand{\ie}{{\em i.e.\ }}
\newcommand{\eg}{{\em e.g.\ }}

\DeclareMathOperator{\E}{\mathbb{E}}     

\renewcommand{\Pr}{}
\let\Pr\relax
\DeclareMathOperator{\Pr}{\mathbb{P}}

\newcommand{\1}[1]{\mathbf{1}_{\set{ #1 } }}

\def\squareforqed{\hbox{\rlap{$\sqcap$}$\sqcup$}}
\def\qed{\ifmmode\squareforqed\else{\unskip\nobreak\hfil
\penalty50\hskip1em\null\nobreak\hfil\squareforqed
\parfillskip=0pt\finalhyphendemerits=0\endgraf}\fi}

\newcommand{\ignore}[1]{ }

\newcommand{\note}[1]{ }

\newcommand{\p}{\partial}

\newcommand{\dist}{\mathrm{dist}}

\newcommand{\AND}{\qquad \textrm{and} \qquad}

\usepackage{capt-of}

\newcommand{\DLA}{\mathsf{DLA}}

\newcommand{\pmin}{\pi_{\min}}
\newcommand{\pmax}{\pi_{\max}}
\newcommand{\tmix}{t_\mathrm{mix}}

\begin{document}

\title{Harmonic Measure in the Presence of a Spectral Gap}
\author[Itai Benjamini]{Itai Benjamini$^1$}
\thanks{$^1$Weizmann Institute of Science.
{\em email:} \texttt{itai.benjamini@weizmann.ac.il} } 
\author[Ariel Yadin]{Ariel Yadin$^2$}
\thanks{$^2$Ben-Gurion University of the Negev. {\em email:} \texttt{yadina@bgu.ac.il}}

\maketitle

\begin{abstract}
We study harmonic measure in finite graphs with an emphasis on 
expanders, that is, positive spectral gap.
It is shown that if the spectral gap is positive then for all sets that are not too large
the harmonic measure from a uniform starting point 
is not more than a constant factor of the uniform measure on the set.
For large sets there is a tight logarithmic correction factor.
We also show that positive spectral gap does not allow for a fixed proportion of the harmonic measure of sets to 
be supported on small subsets, in contrast to the situation in Euclidean space.
The results are quantitative as a function of the spectral gap, and apply also when the spectral gap
decays to $0$ as the size of the graph grows to infinity.
As an application we consider a model of {\em diffusion limited aggregation}, or $\DLA$, on finite graphs, 
obtaining  upper bounds on the growth rate of the aggregate.
\end{abstract}

\section{Introduction}

Given a set of vertices $S$ in a graph, start a random walk from some initial distribution, until it hits $S$.
The probability that the random walk first hits $S$ at a vertex $y$, 
is a probability measure on $S$. This probability 
measure is called the {\em harmonic measure}.

Harmonic measure for Brownian motion in Euclidean space was thoroughly studied with 
spectacular achievements and some fundamental still open problems
(see \eg \cite{Bo, Garnett2005}).
Beyond conformal invariance in two dimensions, the {\it doubling property} or scale invariance was key to the analysis of harmonic measure of subsets of the Euclidean space.

In this note we would like to focus on harmonic measure in the context of finite graphs with small diameter. 
When the graph is rapidly mixing it is natural to expect that the harmonic measure will
be more uniformly spread out. Indeed basic results in this direction are established here.

We consider harmonic measure in the setting of graphs with uniformly bounded 
spectral gap, also known as {\em expander graphs}.
The first main result shows that for subsets that are not too large, 
when starting from the stationary measure,
the harmonic measure 
of a point is at most a constant multiple of the uniform measure.
When the sets in question have large volume, there is a multiplicative logarithmic correction term. 
See Theorem \ref{thm: main}.
This bound is tight as Example \ref{exm: tree expander} shows.
This result may be viewed as a Buerling-type estimate:  it bounds from above the harmonic measure at any point,
showing that no specific vertex in $S$ can carry too much mass of the harmonic measure.

All our results are quantitative, so that they carry over to the case where 
the graphs are not expanders, taking into consideration the asymptotics 
of the spectral gap as the graph size tends to infinity.


\subsection{Support of harmonic measure}

For expander graphs, 
Theorem \ref{thm: no Makarov} shows that for any set in an expander, 
any fixed proportion of the harmonic measure of the set $S$ cannot be supported on small subsets.
It is also shown that this characterizes expander graphs, for a precise statement see 
Proposition \ref{prop: expander characterization}.

The other extreme is the context of polynomial growth, where we believe the following to hold.
Let $(G_n)_n$ be a sequence of finite, connected, vertex transitive graphs with size growing to infinity, 
and uniformly bounded degree, such that $ |G_n | = o( \mathrm{diam}(G_n)^d)$ for some $d>0$.
Then for any set $S_n \subset G_n$ and fixed starting vertex $x_n \in G_n$, 
the harmonic measure of $S_n$ from $x_n$ is supported on a set of size $o(|G_n|)$. 

A strategy to prove this is along the lines of 
adapting the Euclidean case proof \cite{B97} 
using a structure theorem for such graphs \cite{BFT12}.
We plan to pursue this line of thought together with Romain Tessera.

More involved behavior arises for groups which are neither polynomial nor expanders.
One example is when the group  $G$ is  the {\em lamplighter} over $\Z/n\Z$ 
(that is, the group $G = \set{0,1} \wr \Z/n \Z$, see \eg \cite{Woess2005}).  In this case
there is a set $S$ of size proportional to $|G|$ for which harmonic measure is mostly supported on a subset of $S$
of size proportional to $|S|$, obtained by adapting the example in \cite{B97};
namely, let $S$ be the set in which more than $n/2$ of the lamps are on and the lamp at $0$ is also on.  
This suggests a somewhat hybrid  picture for harmonic measure on 
Cayley graphs of super polynomial growth, which are not expanders.

Let us conclude with
\begin{conj}
Assume $G$ is an infinite graph which admits the doubling property; 
that is, there exists a universal constant $C>0$ such that for all $r>0$ and all $x$, 
$| B(x,2r) | \leq C | B(x,r) |$, where $B(x,r)$ is the ball of radius $r$ around $x$ in the graph metric.
Then, as $r \to \infty$, for any subset $S \subset B(x,r)$ and any $z \not\in B(x,r)$, 
$1-o(1)$ of the harmonic measure of $S$ from $z$ is supported on a subset of $S$ of size $o(|B(x,r)|)$.
 \end{conj}

\subsection{$\DLA$}

The study of harmonic measure is key to the still lacking understanding of the $\DLA$ growth process.
It will be of interest to understand harmonic measure and Beurling-type estimates on nilpotent Cayley graphs
(see \cite{Kesten87DLA, Kesten90DLA}; also, Theorem \ref{thm: main} below is a Beurling-type estimate).
In the last section we formulate a Kesten-type result regarding the $\DLA$ aggregate in the presence of positive 
spectral gap, see Theorem \ref{thm: DLA expander}.

\section{Preliminaries and notation}

\subsection{Notation}

We consider a reversible Markov chain on finite state space $G$, 
with transition matrix $P$ and reversing probability measure $\pi$.
We use $(X_t)_t$ to denote the Markov process; \ie
$$ \Pr [ X_{t+1} = y \ | \ X_t = x ] = P(x,y) . $$
$T_S,T_S^+$ to denote the hitting and return times to a set $S$; that is
$$ T_S = \inf \set{ t \geq 0 \ : \ X_t \in S } \AND T_S^+ = \inf \set{ t \geq 1 \ : \ X_t \in S} . $$
For a path $\gamma$, we use $\gamma[s,t]$ to denote the path $(\gamma_s , \ldots, \gamma_t)$.
$\Pr_\mu , \E_\mu$ denote probability measure and expectation conditioned on $X_0$ having distribution $\mu$.
When no starting measure is specified, we refer to starting from the stationary measure $\pi$.
We use $\pmin = \min_x \pi(x)$ and $\pmax = \max_x \pi(x)$.
We denote the spectral gap of $P$ by $1-\lambda$; that is if 
$1 = \lambda_1 > \lambda_2 \geq \cdots \geq \lambda_n  \geq -1$ then $\lambda = \max_{j >1} |\lambda_j|$.

Note that the harmonic measure of any set does not change if we pass from 
$P$ to the {\em lazy} chain $\tfrac12 (I+P)$. The lazy chain has non-negative eigenvalues, and
the spectral gap only changes by a factor of $2$.  
So without loss of generality we will always assume that $\lambda_n > -1$.
That is, throughout this paper we always work with irreducible and aperiodic chains $P$.

By simple random walk on a graph $G$ we refer to the Markov chain whose transition matrix $P$ is given 
by $P(x,y) = \tfrac{1}{\deg(x)} \1{ x \sim y} $. 
When $G$ is a regular finite graph, for the simple random walk $\pi$ is the uniform measure.

$a \vee b$ denotes $\max \set{a,b}$ and $a \wedge b$ denotes $\min \set{a,b}$.

\subsection{Basic facts about hitting times}

This section is a review of known facts, and we include proofs for completeness.

\begin{lem}
\label{lem: expander mixing}
Let $S \subset G$.  Then,
$$ \Pr_\pi [ T_S^+ > t ] \leq \sr{ 1 - (1-\lambda) \pi(S) }^{t/2}  . $$
\end{lem}

\begin{proof}
Variants of this lemma are known, and we include the proof for completeness.
We follow a method from \cite{alon2004probabilistic}.

We consider the space of functions $f:G \to \R$ with inner product
$\IP{f,g} : = \sum_x f(x) g(x) \pi(x)$.
Let $P$ be the transition matrix of the random walk on $G$.
It is well known that because the random walk is reversible with respect to $\pi$,
$P$ is a self-adjoint operator, and thus we may find an orthonormal basis $1 = f_1, f_2, \ldots,f_n$
($|G|=n$) of eigenvectors of $P$, with $P f_j = \lambda_j f_j$ and
$1 = \lambda_1 > \lambda_2 \geq \cdots \geq \lambda_n > -1$.
Recall that $1-\lambda = 1 - \max_{j>1} |\lambda_j|$.

Now, for a set $S \subset G$ let $Q = Q_S$ be the matrix $Q(x,y) = P(x,y) \1{ y \not\in S }$.
It is immediate that
$$ \pi(x) \Pr_x [ T_S^+ > t ] = \pi(x) \sum_{y} Q^t(x,y) = \iP{ Q^t 1 , \delta_x} . $$
Let us bound $\IP{Q f , Q f}$:
Define $\tilde{f}(x) = f(x) \1{x \not\in S}$.
Then $Q f = Q \tilde f = P \tilde f$.
Also, $\iP{\tilde f, \tilde f } \leq \iP{f,f}$.
Thus, for all $f \neq 0$, using the orthonormal decomposition
$ \tilde f = \sum_{j=1}^n \iP{ \tilde f, f_j } f_j $, we obtain:
$ \iP{ \tilde f , \tilde f } = \sum_{j=1}^n \iP{ \tilde f , f_j }^2 , $
and by Cauchy-Schwarz,
\begin{align*}
\iP{ \tilde f , 1 }^2 & = \sr{ \sum_{x \not\in S} \pi(x) f(x) }^2 \leq (1 - \pi(S) ) \cdot \sum_{x \not\in S} \pi(x) f(x)^2
\leq (1- \pi(S) ) \cdot \iP{ \tilde f , \tilde f} .
\end{align*}
Moreover,
\begin{align*}
\iP{ Q f , Q f } & = \iP{ P \tilde f , P \tilde f }  = \sum_{j=1}^n \lambda_j^2 \iP{ \tilde f , f_j}^2
\leq \iP{ \tilde f , 1}^2 + \lambda^2  \sum_{j=2}^n \iP{ \tilde f , f_j}^2 \\
& = \iP{\tilde f , 1 }^2 (1-\lambda^2) + \lambda^2 \iP{ \tilde f , \tilde f }
\\
& \leq
\sr{ 1 - \pi(S) (1-\lambda^2) } \cdot  \iP{ \tilde f , \tilde f } .
\end{align*}
Since $\iP{ \tilde f , \tilde f} \leq \iP{f,f}$ we get that for any $f$ with $\iP{f,f}=1$,
$\iP{Q f , Qf } \leq  1 - \pi(S) (1-\lambda^2) $.  
Another application of Cauchy-Schwarz gives,
$$ \Pr_\pi [ T_S^+ > t ] = \iP{ Q^t 1 , 1 } \leq \sqrt{ \iP{ Q^t 1 , Q^t 1 }  } \leq \sr{ 1 - \pi(S) (1-\lambda) }^{t/2} . $$
\end{proof}

We use the notation $\tmix : = \lceil  \frac{\log (2 / \pmin )}{1-\lambda} \rceil$,
which is convenient because of the next classical proposition.

\begin{prop}
\label{prop: mix}
For any $t \geq \frac{\log (2 / \pmin )}{1-\lambda}$ we have for all $x,y \in G$,
$$ \frac{1}{2} \pi(y) \leq \Pr_x [ X_t = y ] \leq \frac{3}{2} \pi(y) . $$
\end{prop}

\begin{proof}
It is classical, see \eg \cite[Chapter 12]{LevinPeresWilmer}, that
$$ | \Pr_x [ X_t=y ] - \pi(y) | \leq \sqrt{ \tfrac{\pi(y)}{\pi(x) } } \cdot \lambda^t \leq \pi(y) \cdot \frac1{\pmin } e^{- (1-\lambda) t } . $$
\end{proof}

For a Markov chain $P$ we denote $u = u(P) : = \min_{x \neq y} \Pr_x [ T_y < T_x^+ ] $.
If $(P_n)_n$ is a sequence of chains on $G_n$ such that $|G_n| \to \infty$ and $u(P_n)$ is 
bounded away from $0$, then we say that the chains $(P_n)_n$ are {\em uniformly transient}.
Lemma \ref{lem: uniform transience} shows that if the spectral gap and $\tfrac{\pmin}{\pmax}$
are bounded away from $0$, then we have uniform transience.
But uniform transience is a more general property:  indeed, simple random walk on 
$(\Z / n \Z)^d$ for $d \geq 3$ are uniformly transient, a fact arising from the fact that $\Z^d$ 
is transient for $d \geq 3$.

\begin{lem}
\label{lem: uniform transience}
There exists a universal constant $c>0$ such that
for every $x \neq y \in G$,
$$ \Pr_x [ T_y < T_x^+ ] \geq  \frac{c(1-\lambda) \pmin}{\pmax} . $$
That is, $u \geq  \frac{c(1-\lambda) \pmin}{\pmax}$.
\end{lem}

\begin{proof}
The identity
$$ \Pr_x [ T_y < T_x^+ ] = \frac{1}{\pi(x) ( \E_x [ T_y^+ ] + \E_y [ T_x^+ ]) }  $$
is well known, quite simple to prove, and appears \eg in \cite[Chapter 2]{AldousFill}.

Proposition \ref{prop: mix} and the Markov property at time $\tmix$ give us that
$$  \E_x [ T_y^+]
\leq \Pr_x [ T_y^+ \leq \tmix ] \cdot \tmix + \Pr_x [ T_y^+ > \tmix ] \cdot \tfrac32 \cdot \E_\pi [ T_y ] . $$
Starting from the stationary distribution, Lemma \ref{lem: expander mixing} tells us that
$T_y^+$ is dominated by a geometric random variable of mean $\frac{2}{(1-\lambda) \pi(y) }$.
Thus, $\E_x[T_y^+] \leq \tmix \vee \frac{3}{(1-\lambda) \pi(y) }$ and similarly for $\E_y[T_x^+]$.
Altogether,
$$ \Pr_x [ T_y < T_x^+ ] \geq  \frac{c(1-\lambda) \pmin}{\pmax}  . $$
\end{proof}

\section{Harmonic measure from a uniform starting point}

Let $S \subset G$ be some set.
Let
$$ h_{y,S}(x)  : = \Pr_y [ X_{T_S} = x ]  \AND h_S(x) = \sum_y \pi(y) h_{y,S}(x) $$
be the harmonic measure on $S$ from $y$, and from the stationary distribution, respectively.
A simple, but crucial, observation is the following.

\begin{prop}
For any $x \in S , y \in G$,
\begin{align}
\label{eqn: reverse paths}
\pi(y) h_{y,S}(x) &  = \frac{\pi(x) \Pr_x [ T_y < T_S^+ ] }{ \Pr_y [ T_S < T_y^+ ] } .
\end{align}
\end{prop}

\begin{proof}
This is a well known application of path-reversal.
Since the Markov chain is reversible with reversing measure $\pi$,
we have that $\pi(z) P(z,w) = \pi(w) P(w,z)$ for all $w,z \in G$, which leads to
$$ \pi(x_0) \Pr [ X_0=x_0 , \ldots, X_n = x_n ] = \pi(x_n) \Pr [ X_0 = x_n , \ldots, X_n = x_0 ] , $$
for any path $x_0,\ldots,x_n$.

Fix $x \in S \subset G$ and $y \in G \setminus S$.

Let $\Gamma_{y,x,S}$ be all paths 
$\gamma = (\gamma(0),\ldots,\gamma(n))$ in $G$ such that $\gamma(0) = y, \gamma(n)=x$
and $\set{ \gamma(1) , \ldots, \gamma(n-1) } \cap S = \emptyset$;
these are paths that go from $y$ to $x$ never returning to $y$ and hitting $S$ for the first time at $x$.
For a path $\gamma$ let $\Pr[\gamma] = \Pr [ X_0 = \gamma(0) , \ldots, X_{|\gamma|} = \gamma(|\gamma|) ]$.

Summing $\Pr[\gamma]$ over all paths $\gamma$ in $\Gamma_{y,x,S}$ one obtains $\Pr_y [ X_{T_S}=x , T_S < T_y^+ ]$.
Summing $\Pr[\hat \gamma]$ over the reversals $\hat \gamma$ of all paths $\gamma$ in $\Gamma_{y,x,S}$ we get 
$\Pr_x [ T_y < T_S^+ ]$.
Thus, multiplying by $\pi(y)$ or $\pi(x)$ we obtain
$$ \pi(y) \Pr_y [  X_{T_S}=x , T_S < T_y^+ ] = \pi(x) \Pr_x [ T_y < T_S^+ ] . $$

Let $A_k$ be the event that the Markov chain visits $y$ exactly $k$ times up to hitting $S$ (for any integer $k \geq 0$).  Then by the strong Markov property,
\begin{align*}
h_{y,S}(x) & = \Pr_y [ X_{T_S} = x ] = \sum_{k=1}^{\infty} \Pr_y [ X_{T_S} = x , A_k ] \\
& = \Pr_y [ X_{T_S}=x , T_S < T_y^+ ] + \sum_{k=2}^{\infty} \Pr_y [ T_y^+ < T_S ] \cdot \Pr_y [ X_{T_S} = x , A_{k-1} ] \\
& = \Pr_y [ X_{T_S} = x , T_S < T_y^+ ] + \Pr_y [ T_y^+ < T_S ] \cdot h_{y,S} (x) .
\end{align*}
So
$$ h_{y,S}(x) = \frac{\Pr_y [ X_{T_S} = x , T_S < T_y^+] }{ \Pr_y [ T_S < T_y^+ ] }   . $$
This proves the proposition for the case that $y \not\in S$.

Now, if $y \in S$, then $h_{y,S}(x) = \1{x=y}$.
Also, if $y \neq x$, then under $\Pr_x$ we have that $T_y = T_y^+ \geq T_S^+$ and if $y=x$
then under $\Pr_x$ we have $T_y=0 < T_S^+$.
So $\Pr_x [ T_y < T_S^+ ] = \1{y=x}$ and $\Pr_y [ T_y < T_S^+ ]  =1$ for $y \in S$
\end{proof}

\begin{prop}
\label{prop: har bound}
For any $x \in S \subset G$,
$$ 
h_S(x) \leq 
u(P)^{-1} \pi(x) \E_x[T_S^+ ] .  $$
\end{prop}

\begin{proof}
We begin with the fact that
\begin{align} \label{eqn: sum path reversal} 
\sum_y \Pr_x [ T_y < T_S^+ ] & = \sum_y \Pr_x [ y \in X [0, T_S^+ -1 ] ]
= \E_x [ | X[0, T_S^+ - 1] | ] \leq \E_x [ T_S^+ ] ,
\end{align}
since $|X[0,T_S^+ -1] | \leq T_S^+$.
The proof is completed using \eqref{eqn: reverse paths}.
\end{proof}

\begin{thm}
\label{thm: main}
There exist constants $C,C'>0$ such that
for any $x \in S \subset G$,
\begin{align*}
h_S(x) & \leq
\frac{C }{u(P) (1-\lambda)} \cdot \pi(x) \cdot \sr{ \log (2e / \pmin) \vee \pi(S)^{-1} } \\
& \leq \frac{C' \pmax}{(1-\lambda)^2 \pmin} \cdot \pi(x) \cdot \sr{ \log (2e / \pmin) \vee \pi(S)^{-1} } .
\end{align*}

Specifically, if we consider simple random walk on a regular graph,
\begin{align*}
h_S(x) & \leq
\frac{C }{u(P) (1-\lambda)} \cdot \sr{ \frac{1}{|S|} \vee \frac{\log N}{N} } 
\leq \frac{C'}{(1-\lambda)^2} \cdot \sr{ \frac{1}{|S|} \vee \frac{\log N}{N} } .
\end{align*}
\end{thm}

\begin{proof}
By Proposition \ref{prop: mix} with the Markov property at time $\tmix$,
$$ \E_x [ T_S^+] \leq \tmix \cdot \Pr_x [ T_S^+ \leq \tmix ] + \Pr_x [ T_S^+ > \tmix ] \cdot \frac{3}{2} \cdot \E_\pi [ T_S^+] . $$
Lemma \ref{lem: expander mixing} implies that starting from the stationary distribution $T_S^+$ is
dominated by a geometric random variable of mean $\frac{2}{\pi(S) (1-\lambda)}$.
So,
$$ \E_x [T_S^+ ] \leq \tmix \vee \frac{3}{\pi(S) (1-\lambda)} \leq
\frac{3}{1-\lambda} \cdot \sr{ \log (2e / \pmin) \vee \frac{1}{\pi(S)} }   . $$
Thus,  the upper bound in  Proposition \ref{prop: har bound} completes the proof of the first inequality.
The second inequality comes from plugging in the lower bound on $u=u(P)$ in Lemma \ref{lem: uniform transience}.
\end{proof}

Theorem \ref{thm: main} and Proposition \ref{prop: har bound} are tight up to constants as the following example shows.

\begin{exm}
\label{exm: tree expander}
Let $G$ be the graph obtained by taking a depth $k$ binary tree and connecting the $2^k$ leaves with extra edges coming from
a $3$-regular graph on $2^k$ vertices with spectral gap $1-\lambda$.
The Markov chain we consider is the simple random walk on $G$.

It is simple to verify that the spectral gap of this walk is just a function of
$1-\lambda$ above, and specifically is bounded away from $0$ independently of $k$
(one way is to verify that the linear isoperimetric inequality holds, and use Cheeger's inequality).
The maximal degree is $4$ and minimal degree is $2$, so
$\pmax = 2 \pmin$.  Now consider the set $S$ consisting of all the leaves of the original tree and the root of the tree.
Let $x$ be the root of the tree.
A random walk starting at $x$ will hit the leaves of the tree before returning to $x$ with probability at least 
the escape probability in the infinite-depth binary tree.
Since the distance from $x$ to the leaves is $k$, we have that $\Pr_x [ |X[0,T_S^+-1] | \geq k ] \geq \alpha > 0$
for some $\alpha$ independent of $k$.  Thus, $\E_x [ |X[0,T_S^+-1] | ] \geq \alpha k$. 
Using the equalities in \eqref{eqn: sum path reversal} to sum \eqref{eqn: reverse paths} we have that
$$ h_S(x) \geq \pi(x) \E_x[ |X[0, T_S^+-1] | ] \geq \frac{c k }{|S|} , $$
for some constant $c>0$ independent of $k$.
This lower bound matches the upper bound in Theorem \ref{thm: main} up to constants
(as mentioned above, $\frac{1}{1-\lambda} = O(1)$).

(One may wish to restrict to connected subsets of $G$, but this is similarly analyzed,
since we could have chosen $S$ to be the
leaves together with a simple path from $x$
to the leaves, and the analysis would still be the same -
it only depended on the fact that with probability bounded away from $0$ a random walk
starting at $x$ will reach the leaves before returning to $S$.)

Thus, the $\log |G|$ (or, rather, $\log (2e/\pmin)$) term in 
Theorem \ref{thm: main} cannot be removed in the general case.
\end{exm}

\begin{rem}
Another question that arises when considering Theorem \ref{thm: main}, is whether a similar result 
could hold for the harmonic measure starting from a fixed typical point, not just from the stationary distribution.
However, this does not hold. To see this, consider simple random walk on a transitive $d$-regular graph $G$.

Let $0 < r < \mathrm{diam}(G)$.  Suppose that $S$ is a set such that $G$ is contained in the $r$-neighborhood of $S$;
\ie $G = \set{ y \ : \ \dist(y,S) \leq r }$.
Then, for any $y \in G$ there exists $x \in S$ such that $\dist(y,S) = \dist(y,x) \leq r$. 
Thus, $h_{y,S}(x) \geq d^{-r}$.

Now, if $d^{-r} > \frac{1}{\eps |S|}$ for some small $\eps>0$, 
we have that for every $y \in G$ there exists $x \in S$ with harmonic 
measure significantly larger than $|S|^{-1}$.

Let $b$ be the size of the ball of radius $r/2$ in $G$.  We may choose a collection of disjoint balls
of radius $r/2$, $B_1,\ldots B_k$ so that for $B = \bigcup_{j=1}^k B_j$ we have 
$G = \set {y \  : \ \dist(y,B) \leq r/2}$.  Since these balls are disjoint we have that $k \leq \frac{|G|}{b}$.
Also, if $y \in G$ then there exists $1 \leq j \leq k$ such that $\dist(y,B_j) \leq r/2$.
Thus, if $S = \set{ x_1,\ldots,x_k}$ where $x_j$ is the centre of the ball $B_j$, then $|S| \leq \frac{|G|}{b}$
and $G = \set{ y \ : \ \dist(y,S) \leq r }$.

Thus, if $|G| > \tfrac1{\eps} d^{3r/2}$ then we may choose $|S| \leq \frac{|G|}{b}$ but also
such that for every $y \in G$ there exists $x \in S$ with $h_{y,S}(x) > \frac{1}{\eps |S|}$.

So for any $\eps > 0$, there are many graphs $|G|$ such that we may find $S \subset G$
so that for any $y \in G$ there exists $x \in S$ with $h_{y,S} (x) > \tfrac{1}{\eps |S|}$.

That is, the use of the stationary measure as the starting measure in Theorem \ref{thm: main} is crucial.
\end{rem}

\section{No small support for expanders}

\subsection{Support of harmonic measure}
A theorem of Makarov \cite{Makarov} states that the harmonic measure of a simply connected 
domain in the plane is supported
on small subsets (in fact sets of dimension $1$, see \cite{JW88}).
Lawler has shown the analogous result for random walk in discrete space \cite{Law93}.
Theorem \ref{thm: main} above tells us that large sets $S$ may have points
that attract a lot of harmonic measure, perhaps up to a logarithmic factor more.
One may {\em a-priori} think that perhaps there are enough such points so that
the harmonic measure will be supported on a very small subset of $S$, similarly to the case
of Makarov's Theorem in the plane.
However, the following theorem shows that the harmonic measure cannot
be supported on small subsets of $S$.

\begin{thm}
\label{thm: no Makarov}
There exists a constant $C>0$ such that for all small $\eps>0$ the following holds.
Suppose that $A \subset S \subset G$ such that $\pi(A) \leq \eps \pi(S)$.
Then,
$$ h_S(A) = \sum_{x \in A} h_S(x) \leq C \frac{\eps \log \tfrac{1}{\eps} }{1-\lambda}  . $$
Specifically, in the case of simple random walk on a regular graph $G$,
if $A \subset S \subset G$ is such that $|A| \leq \eps |S|$ then the above bound on $h_S(A)$ holds.
\end{thm}

\begin{proof}
Write $B = S \setminus A$.
Let $K>0$ be some constant and let $M = \big\lceil \frac{K}{\pi(B)} \big\rceil$.
The event $T_A < T_B$ implies that
either $T_A < M$ or $T_B > M$.
The probability of the former is bounded by
\begin{align*}
\Pr_{\pi}  \big[ T_A < M \big] & \leq \sum_{t=0}^{ M- 1 } \Pr_{\pi}  [ X_t \in A ] =
M \cdot \pi(A) \leq K \cdot \frac{\pi(A)}{\pi(B) } + \pi(A) .
\end{align*}
By Lemma \ref{lem: expander mixing},
\begin{align*}
\Pr_{\pi}  \big[ T_B > M \big] & \leq \exp \sr{ - \tfrac{1-\lambda}{2} \pi(B) \cdot M }
\leq \exp \sr{ - \tfrac{1-\lambda}{2} K } .
\end{align*}
If we take $K = \tfrac{2}{1-\lambda} \log \sr{ \tfrac{ \pi(B) }{\pi(A) } }$
we obtain
$$ \Pr_{\pi}  [ T_A < T_B ] \leq \frac{2}{1- \lambda} \cdot \frac{\pi(A) }{ \pi(B) } \cdot
\log \sr{ \tfrac{\pi(B) }{ \pi(A) } } + \pi(A) + \frac{\pi(A) }{\pi(B) } . $$
Thus, if $\pi(A) \leq \eps \pi(S)$ then
$$ \Pr_{\pi}  [ T_A < T_B ] \leq C \cdot \tfrac{1}{1-\lambda} \cdot \eps \cdot \log \tfrac1{\eps} , $$
for some constant $C>0$.
\end{proof}

\begin{rem}
Note that Theorem \ref{thm: no Makarov} is tight in the following sense: 
If we choose $S$ and $x \in S$ such that $|S| = 2^k$ and $h_S(x) \geq c \tfrac{k}{|S|}$ 
as in Example \ref{exm: tree expander}, then with $A= \set{x}$ we have $\eps = 2^{-k}$, and we 
see that the $\log \tfrac{1}{\eps}$ factor in Theorem \ref{thm: no Makarov} cannot be removed without
further assumptions.
\end{rem}

Let $G = (\Z / n \Z)^d$ for $d \geq 3$ be the $d$-dimensional torus.  
As remarked in the introduction, the harmonic measure of a subset in $G$ 
is supported on small sets.  
However, since Theorem \ref{thm: no Makarov} is quantitative, we can bound the
size of the support.

\begin{cor}
There exists a constant $C>0$ so that the following holds for all small $\eps>0$.
Let $n$ be large enough and let $G = (\Z / n \Z)^d$ for $d \geq 3$.
Then, for any set $S \subset G$ we have that 
if $A \subset S$ satisfies $|A| \leq \tfrac{\eps}{n^2 \log n} |S|$
then the harmonic measure of $A$ is bounded by 
$h_S(A) \leq C \eps$.
\end{cor}

\begin{proof}
It is well known that for $G = (\Z / n \Z)^d$ the spectral gap 
is bounded by $1-\lambda \geq c n^{-2}$ for some constant $c>0$ (depending only on the dimension $d$).
Also, for $d \geq 3$ the $d$-dimensional tori $(\Z/n\Z)^d$ are uniformly transient (this follows from 
the fact that $\Z^d$ is transient for $d \geq 3$).
The corollary now follows by plugging this into Theorem \ref{thm: no Makarov}.
\end{proof}

\subsection{A characterization of expanders}

Theorem \ref{thm: no Makarov} shows that for a sequence of expander graphs $(G_n)_n$,
for any set, it is not possible for small subsets to carry $\frac12$ of the harmonic measure.
Anna Erschler asked if this characterizes expander graphs.  Indeed this is the content of 
this subsection.

Let us first define two quantities associated with a reversible Markov chain $P$ on finite states space $G$.
For a subset $S \subset G$ define 
$$ \p S : = \set{ x \in S \ : \  \exists \ y \not\in S \ , \ P(y,x) > 0 } \AND S^{\circ} = S \setminus \p S .  $$
$\p S$ are the sites accessible from outside of $S$ by one step of the Markov chain.
If $\pi$ is the reversing measure for $P$, define
$$ \Phi = \Phi(G) : = \min_{\substack{ S \subset G \\ 0 < \pi(S) \leq \tfrac12  } } \frac{ \pi(\p S) }{ \pi (S) } . $$
(This is the so called {\em Cheeger constant}.)
It is immediate that $\Phi \in (0,1]$ and it is well known that $C^{-1} \Phi^2 \leq 1- \lambda \leq C \Phi$
for some universal constant $C>0$, and $1-\lambda$ the spectral gap of $G$.

For a set $S \subset G$ define 
$$ \beta_S : = \min_{h_S(A) \geq \tfrac12 } \frac{\pi(A) }{ \pi(S) } \AND 
\beta = \beta(G) : = \min_{\pi(S) \leq \tfrac12 } \beta_S . $$ 
Of course $\beta \in (0,1]$.
Note that with this definition the content of Theorem \ref{thm: no Makarov} is that 
$\beta \log \tfrac{1}{\beta} \geq \tfrac1{2 C} (1-\lambda)$.

Thus, for a sequence of expander graphs $(G_n)_n$, the sequence $\beta(G_n)$ is uniformly 
bounded away from $0$.
The following theorem provides a complementary bound. 
Specifically it shows that $(G_n)_n$ is a sequence of expanders if and only if the sequence 
$(\beta(G_n))_n$ is uniformly bounded away from $0$.

\begin{prop}
\label{prop: expander characterization}
We have $\beta(G) \leq \Phi(G)$.
Consequently, $(G_n)_n$ is a sequence of graphs with $\inf_n \beta(G_n) > 0$
if and only if $(G_n)_n$ is a sequence of expander graphs.
\end{prop}

\begin{proof}
Let $S \subset G$ be a set such that 
$\pi(S) \leq \tfrac12$ and $\pi(\p S)  = \Phi \cdot \pi(S)$ (a Folner set). 
Note that for all $y \in S^{\circ}$ we have $h_{y,S}(x) = \1{x=y}$, so $h_{y,S}(\p S) = 0$,
and for all $y \not\in S^{\circ}$ we have $h_{y,S}(\p S) = 1$. Thus,
$$ h_S(\p S) = 1 - \pi(S^{\circ} ) = 1 - \pi(S) + \pi(\p S) = 1 - (1-\Phi) \pi(S)  . $$
So
$$ h_S(\p S) \geq \frac{1 + \Phi}{2} > \tfrac12 . $$
Thus,
$$ \beta \leq \beta_S \leq \tfrac{\pi(\p S) }{ \pi(S) } = \Phi  . $$

Finally, since $\beta \log \frac1{\beta} \geq c (1-\lambda)$ by Theorem \ref{thm: no Makarov}, 
we get that for a sequence of graphs $(G_n)_n$, the spectral gap is uniformly bounded away from $0$
if and only if $\inf_n \beta(G_n) > 0$.
\end{proof}

\begin{rem}
It is worth noting that for non-expanders one may also find sets such that $\frac12$ the harmonic
measure is supported on subsets of the {\em boundary} that are much smaller than the boundary itself
(not just much smaller than the set).

For example, in a $d$-regular graph $G$, 
if $S \subset G$ has $|\p S| = \Phi \cdot |S|$ and $|S| \leq \frac{|G|}{2}$ (a Folner set),
we may augment $S$ by removing $k = \lfloor \frac{|S^{\circ} |}{d+1} \rfloor$ isolated vertices from $S^{\circ}$,
so that the resulting set $R$ has $|\p R| \geq k + |\p S| \geq c  |S|$,
where $c>0$ depends only on the degree $d$.

However, since the vertices removed are from the interior $S^{\circ}$, 
we still have that for $y \not\in S^{\circ}$ the harmonic measure of $R$ is supported on $\p S \subset \p R$.
So $h_R(\p S) \geq \frac{|G| - |S^{\circ} |}{|G|} > \frac12$.
Also, $|\p S| = \Phi \cdot |S| \leq \Phi \cdot c^{-1} | \p R | $.

Thus, if $(G_n)_n$ is a non-expander sequence, we may find $S_n \subset G_n$
and $A_n \subset \p S_n$ such that 
$h_{S_n}(A_n) > \frac12$ for all $n$ and $\frac{|A_n|}{|\p S_n| } \to 0$ as $n \to \infty$.
\end{rem}

\section{An application to $\DLA$ on expanders}

Diffusion Limited Aggregation, or $\DLA$, is a model introduced by Witten \& Sander \cite{WittenSander}
in which particles are aggregated using the harmonic measure from infinity;
that is, at each time step a particle is released from infinity in $\Z^d$, and performs a random walk 
until hitting the existing aggregate.  Once hitting the aggregate it sticks to the first position it hits.
This model has long resisted rigorous analysis and is considered a very difficult.
Perhaps the only notable result is a bound of Kesten \cite{Kesten87DLA, Kesten90DLA} that shows that
the growth rate of the $\DLA$ aggregate is not too rapid.
Kesten utilizes a discrete Beurling estimate:  he shows that the harmonic measure 
of any point in a connected subset in $\Z^d$ of some diameter cannot be too large.
He then obtains a lower bound on the time it takes a $\DLA$ aggregate to reach distance $r$
using this estimate.
(For more on harmonic measure from infinity, Beurling estimates and $\DLA$ see also \cite{LawlerIntersections}.)
Being such a difficult model to analyze, other variants of $\DLA$ have been considered.
Examples in the non-amenable (\ie expanding) setting include \cite{BPP, Eldan}.

Let us define $\DLA$ properly in our setup: the finite graph case.

\begin{dfn}
Let $G$ be a finite graph and fix $s,e \in G$ as start and end vertices.
Diffusion Limited Aggregation, or $\DLA$, on $G$
is the process $\set{s} = A_0 \subset A_1 \subset A_2 \subset \cdots \subset G$ defined as follows:

Start with $A_0 = \set{s}$.  At each time step $t>0$ let $a_t$ be a random vertex with distribution given
by $h_{\p A_{t-1} }$; that is, $a_t$ is the first point in $\p A_{t-1}$ hit by a random walk started from stationarity.
Set $A_t = A_{t-1} \cup \set{a_t}$.

Stop the process at time $\tau = \inf \set{ t  \ : \  e \in A_t }$.
We use the convention that $A_t = A_\tau$ for all $t > \tau$.
\end{dfn}

We now proceed to prove a lower bound on the volume of the final aggregate in $\DLA$ on a finite graph,
which is an upper bound on the speed the aggregate grows.
We first consider the case of expander graphs, \ie those with bounded spectral gap.

First an auxiliary large deviations calculation:
\begin{lem}
\label{lem: LD}
Let $B = \sum_{n=1}^k Z_n$ for independent Bernoulli random variables $(Z_n)_n$,
each of mean $\E Z_n = p_n$.  Then, for any $C>1$ we have
$$ \Pr [ B \geq C \E B ] \leq \exp \sr{ - \E B \cdot C \log (C/e)  } . $$
\end{lem}

\begin{proof}
We use the well known method by Bernstein.
For $\alpha > 0$ we may bound the exponential moment of $B$ as follows:
\begin{align*}
\E e^{\alpha B} & = \prod_{n=1}^k \E e^{\alpha Z_n} = \prod_{n=1}^k ( (e^{\alpha} -1) p_n + 1 )
\leq \exp \sr{ ( e^{\alpha} -1) \E B } .
\end{align*}
By Markov's inequality,
$$ \Pr [ B \geq C \E B ] = \Pr [ e^{\alpha B} \geq e^{\alpha C \E B} ] \leq \exp \sr{ ( e^{\alpha} - 1) \E B - \alpha C \E B } . $$
So we wish to minimize the term $e^{\alpha} - 1 - \alpha C$ over positive $\alpha$.
Taking derivatives this is minimized when $e^{\alpha} = C$ (recall that $C>1$), so
$$ \Pr [ B \geq C \E B ] \leq \exp \sr{ \E B \cdot (C-1 - C \cdot \log C ) } . $$
\end{proof}

\begin{thm}
\label{thm: DLA expander}
Let $(G_n)_n$ be a sequence of expander graphs (\ie the spectral gap $1-\lambda$ is uniformly bounded below)
of maximal degree $d$.
For every $n$ let $s,e$ be vertices realizing the diameter of $G_n$, and
consider $\DLA$ on $G_n$ starting at $s$ and ending when first absorbing $e$.
Then, with probability tending to $1$ as $n \to \infty$, the final $\DLA$ aggregate will contain
at least $|G_n|^c$ particles, where $c>0$ is some constant (independent of $n$).
\end{thm}

\begin{proof}
We adapt an argument of Kesten, see \cite{Kesten87DLA, Kesten90DLA}.

For a self-avoiding path $v_1,v_2,\ldots,v_m$ in $G$ we say that $v_1,\ldots,v_m$ are {\em filled in order}
if there exist $0 \leq t_1 < t_2 < \cdots < t_m \leq \tau$ such that $a_{t_j} = v_j$ for all $j$ (where $a_0=s$).
(on this event it may be that particles stick to other vertices in between $v_j, v_{j+1}$, but it cannot be that 
a particle sticks to $v_j$ before some particle sticks to $v_i$ for $j>i$).

Let $r(t) = \max_{x \in A_t} \dist(x,s)$ be the diameter of the aggregate at time $t$.
Note that $r(\tau) = \dist(e,s)$.

If for some $k > 0$ we have $r(t+k) = r(t) + m$, there must exist 
a self avoiding path $v_0,v_1,\ldots,v_m$ such that $v_0 \in A_t$, $\dist(v_0,s) = r(t)$, 
$v_m \in A_{t+k}$,
and $v_0,\ldots,v_m$ is filled in order.

The number of choices for such path $v_0,v_1,\ldots,v_m$ is at most $|A_t| \cdot d^m = t d^m$,
where $d$ is the maximal degree in $G$.

Fix some such self avoiding path $v_0,v_1,\ldots,v_m$.
For every $t+1 \leq n \leq t+k$ define $u_n$ to be the unique vertex $v_j$ such that
$\set{ v_0,v_1 , \ldots, v_{j-1} } \subset A_{n-1}$ and $v_j \not\in A_{n-1}$.  That is, $u_n$ is the upcoming vertex
in the path $v_1,\ldots,v_m$ that needs to be filled by the $\DLA$ process.

As long as $|A_t| = t \leq \frac{|G|}{\log (2e |G|) }$ we have that
$h_{A_t}(x) \leq C (1-\lambda)^{-2} \cdot \tfrac{1}{|A_t| } = \frac{C}{(1-\lambda)^2 t}$,
by Theorem \ref{thm: main}.
Thus, if we define
$$ I = \sum_{n=t+1}^{t+k} \1{ a_n = u_n } \AND B = \sum_{n=t+1}^{t+k} Z_{n-1} , $$
we have that $I$ is stochastically dominated by $B$,
where
$(Z_n)_n$ are independent Bernoulli random variables with mean $\E[Z_n] = C (1-\lambda)^{-2} n^{-1}$.
(If $t+k \geq \tau$ then some of the indicators in the sum for $I$ are $0$.)
However, in order for $v_0,v_1,\ldots,v_m$ to be filled in order we must have
that $I \geq m$.  Thus, using Lemma \ref{lem: LD}, if $m = C \E B$ for some $C>1$,
the probability that $v_1,\ldots,v_m$ are filled in order is bounded by
\begin{align*}
\Pr [ I \geq m ] & \leq \Pr [ B \geq m ] \leq \exp \sr{ - m \log (C/e) } .
\end{align*}
Thus, taking $C'>1$ large enough (depending on $d$),
if we sum over all possible choices for the path $v_1,\ldots, v_m$,
we obtain that
$$ \Pr [ r(t+k) - r(t) \geq m ] \leq e^{ - m \log (C'/e) } \cdot t d^m \leq e^{- c m} . $$
for some constant $c>0$.
Since
$$ \E B = \tfrac{C}{(1-\lambda)^2} \sum_{n=0}^{k-1} \tfrac1{t+n} \leq \frac{C}{(1-\lambda)^2} \cdot \log (1 + k/t) , $$
we have that for some small enough constant $c>0$,
with $k = e^{ c (1-\lambda)^2 \dist(s,e) } = o(\frac{|G|}{\log (2e |G|) } )$,
$$ \Pr [ r(k) - r(1) \geq  \dist(s,e) - 1 ] \leq \exp \sr{- c' \dist(s,e) } . $$
Since $1-\lambda = \Theta(1)$ and $\dist(s,e) = \Theta(\log |G_n|)$ as $n \to \infty$ we get that
with high probability the $\DLA$ aggregate stops after $e^{c (1-\lambda)^2 \dist(s,e) } = \Omega(|G_n|^{c'})$ 
particles.
\end{proof}

\begin{ques}
For a sequence of expander graphs $(G_n)_n$, let $s,e$ be vertices realizing the diameter of $G_n$,
and consider $\DLA$ starting at $s$ and ending at $e$.
Is it true that with probability tending to $1$ as $n \to \infty$, the final $\DLA$ aggregate on $G_n$ will contain
$c |G_n|$ particles, $c>0$ a constant independent of $n$?
\end{ques}

%
%

%

\bibliographystyle{plain}


\end{document}